\documentclass[reqno,a4paper,12pt]{amsart} 
\usepackage{hyperref} 
\usepackage{amssymb} 
\usepackage{fullpage} 
\usepackage{latexsym} 
\usepackage{amsmath} 
\usepackage{esint} 
\usepackage[utf8x]{inputenc}

\allowdisplaybreaks

\newtheorem{theorem}{Theorem}[section]

\newtheorem{lemma}[theorem]{Lemma}

\theoremstyle{definition} 
\newtheorem{definition}[theorem]{Definition} 
\newtheorem{remark}[theorem]{Remark}

\numberwithin{theorem}{section} \numberwithin{equation}{section}

\newcommand{\R}{{\mathbb R}}

\newcommand{\B}{{\mathbf B}}

\newcommand{\grad}{\nabla}

\newcommand{\tint}{\int_{t_1}^{t_2} \int_{\Omega}} 
\newcommand{\tauint}{\int_{\tau_1}^{\tau_2} \int_{\Omega}} 
\newcommand{\tsup}{\esssup_{t_1 < t < t_2} \int_{\Omega}}

\newcommand{\dx}{\, dx} 
\newcommand{\dt}{\, dt} 
\newcommand{\dxdt}{\dx\dt}

\DeclareMathOperator*{\essliminf}{ess\,lim\,inf} 
\DeclareMathOperator*{\essinf}{ess\,inf} 
\DeclareMathOperator*{\esssup}{ess\,sup}

\newcommand{\dif}[0]{\ensuremath{\,\mathrm{d}}} 
\newcommand{\norm}[1]{\ensuremath{\Vert #1 \Vert}}

\newcommand{\abs}[1]{\ensuremath{\vert #1 \vert}}

\newcommand{\lscss}{semicontinuous supersolution}

\begin{document} 
\title[Lower semicontinuity of supersolutions]{Lower semicontinuity of weak supersolutions to the porous medium equation}

\address{Benny Avelin\\Department of Mathematics, Uppsala University\\
S-751 06 Uppsala, Sweden} \email{benny.avelin@math.uu.se}

\address{Teemu Lukkari\\
Department of Mathematics and Statistics, University of Jyväskylä\\
P.O. Box 35, 40014 Jyväskylä, Finland} \email{teemu.j.lukkari@jyu.fi} 
\author{Benny Avelin and Teemu Lukkari} 



\thanks{The research reported in this work was done during the
  authors' stay at the Institut Mittag-Leffler (Djursholm, Sweden).}

\maketitle 
\begin{abstract}
  Weak supersolutions to the porous medium equation are defined by
  means of smooth test functions under an integral sign. We show that
  nonnegative weak supersolutions become lower semicontinuous after
  redefinition on a set of measure zero. This shows that weak
  supersolutions belong to a class of supersolutions defined by a
  comparison principle.
\end{abstract}

\maketitle

\section{Introduction}

We study regularity properties of weak supersolutions to the porous medium equation 
\begin{displaymath}
	\frac{
	\partial u}{
	\partial t}-\Delta u^m=0, 
\end{displaymath}
where $m>1$. This equation has attracted a lot of attention during the
last decades, mostly because of its interesting mathematical
properties. This equation shares many properties with the so called
$p$-parabolic equation, for example, intrinsic scaling and finite
speed of propagation. However, the porous medium equation is a
different game. When studying finer properties of the equation, the
techniques often differ although the results are essentially the
same. Since $m>1$, the equation is degenerate, i.e. the modulus of
ellipticity vanishes when the solution is zero. For more information
about this type of equations, including numerous further references,
we refer to the monographs \cite{DaskalopoulosKenig} and
\cite{VazquezBook}.

Weak supersolutions to the porous medium equation are defined via a
variational inequality: they satisfy
\begin{displaymath}
	\int_{\Omega_T}-u\frac{
	\partial \varphi}{
	\partial t}+\nabla u^m\cdot\nabla \varphi\dif x\dif t\geq 0 
\end{displaymath}
for all nonnegative smooth test functions $\varphi$ with compact
support. On the other hand, in potential theory, it is natural to
consider a notion of supersolutions defined via the comparison
principle. This means that a lower semicontinuous function $u$ is a
\emph{\lscss} if it obeys the parabolic comparison principle with
respect to continuous solutions. In the classical potential theory of
the Laplacian, this definition is due to Fr{\'e}d{\'e}ric Riesz, see
\cite[pp. 333]{Riesz}. Observe that \lscss s are defined in \emph{every}
point.  For the porous medium equation, see \cite{KinnunenLindqvist2},
where the label viscosity supersolutions is used.

The natural question is now what is the exact relationship between the
two classes of supersolutions. The expectation is that weak
supersolutions should enjoy ``one sided'' regularity (lower
semicontinuity) since solutions have ``two sided'' regularity
(continuity). Not only is the lower semicontinuity of weak
supersolutions interesting in its own right, but in classical
potential theory lower semicontinuity plays a key role, connecting the
variational formulation with the potential theoretic one.

Our main result shows that nonnegative weak supersolutions indeed are,
up to a choice of a proper pointwise representative, lower
semicontinuous. That is we prove
\begin{theorem}
  \label{thm:lsc:intro} Let $u$ be a nonnegative weak supersolution to
  the porous medium equation in $\Omega\times(t_1,t_2)$. Then
	\begin{displaymath}
		u(x,t)=\essliminf_{(y,s) \to (x,t)} u(y,s) = \lim_{r \to 0} \essinf_{(x,t)+Q(r,r^2)} u 
	\end{displaymath}
	at all Lebesgue points of $u$ such that $u(x,t)<\infty$, where $Q(r,r^2) = B(0,r) \times (-r^2,r^2)$. In particular, $u$ has a lower semicontinuous representative. 
\end{theorem}

Theorem \ref{thm:lsc:intro}, together with the comparison principle
between weak supersolutions and weak solutions, shows that weak
supersolutions are also {\lscss s}. In the other direction it was
proved in \cite{KinnunenLindqvist2} that locally bounded {\lscss s}
are weak supersolutions. Thus the two different notions are
coherent. One should also note that the class of {\lscss s} is
strictly larger if unbounded functions are allowed. To see this,
consider the celebrated Barenblatt solution, \cite{Barenblatt,
  ZeldovichKompaneets},
\begin{displaymath}
	\B_m(x,t)= 
	\begin{cases}
		t^{-\lambda}\left(C-\frac{\lambda(m-1)}{2mn} \frac{\abs{x}^2}{t^{2\lambda/n}}\right)_+^{1/(m-1)}, & t>0,\\
		0, & t\leq 0, 
	\end{cases}
\end{displaymath}
where $\lambda=n/(n(m-1)+2)$, and $C>0$ can be chosen freely. It is not a weak supersolution, since its gradient fails to have the required amount of integrability, i.e. $|\nabla \B_m^m|$ is not in $L^2_{loc}(E)$ for any open set $E$ containing the origin. However, the Barenblatt solution still obeys the comparison principle with respect to continuous solutions and thus it is a \lscss.

To prove that a weak supersolution has a lower semicontinuous
representative, we adapt the ideas used for a class of equations
containing the $p$-parabolic equation, see \cite{Kuusi}. The main
technical tool in \cite{Kuusi} is an $L^\infty$ estimate for weak
subsolutions. The chief difficulty in adapting the arguments is that
we may not add constants to subsolutions to the porous medium
equation, i.e. $(k-u)_+$ is in general not a subsolution when $u$ is a
weak supersolution. We deal with this by proving the necessary
$L^\infty$ estimates directly for $(k-u^m)_+$, by a version of the De
Giorgi iteration procedure.

The paper is organized as follows. In Section \ref{sec:weaksuper}, we recall the exact definition of weak supersolutions, {\lscss s} and some technical results needed for the estimates. In Section \ref{sec:energy_estimate}, we derive an energy estimate for truncations of weak supersolutions. This estimate is then used in a De Giorgi type iteration process in Section \ref{sec:boundedness} to get the $L^\infty$ estimate needed in the lower semicontinuity proof. Finally, Section \ref{sec:lower_semicontinuity} contains the proof of our main result, Theorem \ref{thm:lsc:intro}.

\section{Supersolutions} \label{sec:weaksuper}

Let $\Omega$ be an open and bounded subset of $\R^N$, and let $0<t_1<t_2<T$. We use the notation $\Omega_T=\Omega\times(0,T)$ and $U_{t_1,t_2}=U\times (t_1,t_2)$, where $U\subset\Omega$ is open. The parabolic boundary $
\partial_p U_{t_1,t_2}$ of a space-time cylinder $U_{t_1,t_2}$ consists of the initial and lateral boundaries, i.e. 
\begin{displaymath}
	\partial_p U_{t_1,t_2}=(\overline{U}\times\{t_1\}) \cup (
	\partial U\times [t_1,t_2]). 
\end{displaymath}
The notation $U_{t_1,t_2}\Subset\Omega_T$ means that the closure $\overline{U_{t_1,t_2}}$ is compact and $\overline{U_{t_1,t_2}}\subset\Omega_T$.

We use $H^1(\Omega)$ to denote the usual Sobolev space, the space of functions $u$ in $L^2(\Omega)$ such that the weak gradient exists and also belongs to $L^2(\Omega)$. The norm of $H^1(\Omega)$ is 
\begin{displaymath}
	\norm{u}_{H^1(\Omega)}=\norm{u}_{L^2(\Omega)}+\norm{\nabla u}_{L^2(\Omega)}. 
\end{displaymath}
The Sobolev space with zero boundary values, denoted by $H^{1}_0(\Omega)$, is the completion of $C^{\infty}_0(\Omega)$ with respect to the norm of $H^1(\Omega)$.

The parabolic Sobolev space $L^2(0,T;H^1(\Omega))$ consists of measurable functions $u:\Omega_T\to[-\infty,\infty]$ such that $x\mapsto u(x,t)$ belongs to $H^1(\Omega)$ for almost all $t\in(0,T)$, and 
\begin{displaymath}
	\int_{\Omega_T}\abs{u}^2+\abs{\nabla u}^2\dif x\dif t<\infty. 
\end{displaymath}
The definition of $L^2(0,T;H^{1}_0(\Omega))$ is identical, apart from the requirement that $x\mapsto u(x,t)$ belongs to $H^{1}_0(\Omega)$. We say that $u$ belongs to $L^2_{loc}(0,T;H^{1}_{loc}(\Omega))$ if $u\in L^2(t_1,t_2;H^1(U))$ for all $U_{t_1,t_2}\Subset\Omega_T$.

Supersolutions to the porous medium equation are defined in the weak sense in the parabolic Sobolev space. 
\begin{definition}
	\label{def:local-weak} A nonnegative function $u:\Omega_T\to\R$ is a \emph{weak supersolution} of the equation 
	\begin{equation}
		\label{eq:pme} \frac{
		\partial u}{
		\partial t}-\Delta u^m=0 
	\end{equation}
	in $\Omega_T$, if $u^m\in L^2_{loc}(0,T;H^{1}_{loc}(\Omega))$ and 
	\begin{equation}
		\label{eq:weak-pme} \int_{\Omega_T}-u\frac{
		\partial\varphi}{
		\partial t} +\nabla u^m\cdot\nabla\varphi\dif x\dif t\geq 0 
	\end{equation}
	for all positive, smooth test functions $\varphi$ compactly supported in $\Omega_T$. The definition of \emph{weak subsolutions} is similar; the inequality is simply reversed. \emph{Weak solutions} are defined as functions that are both super- and subsolutions. 
\end{definition}

Weak solutions are locally H\"older continuous, after a possible
redefinition on a set of measure zero. See \cite{DK},
\cite{DaskalopoulosKenig}, \cite{DiBeFried}, \cite{VazquezBook}, or
\cite{Kiinalaiset}.

Our main aim in this note is to relate the notion of weak
supersolutions to the following class of supersolutions.  
\begin{definition}
	\label{def:viscosity-supersols} A function $u:\Omega_T\to [0,\infty]$ is a \emph{\lscss }, if 
	\begin{enumerate}
		\item $u$ is lower semicontinuous, 
		\item $u$ is finite in a dense subset of $\Omega_T$, and 
		\item the following parabolic comparison principle holds: Let $U_{t_1,t_2}\Subset\Omega$, and let $h$ be a solution to \eqref{eq:pme} which is continuous in $\overline{U_{t_1,t_2}}$. Then, if $h\leq u$ on $
		\partial_p U_{t_1,t_2}$, $h\leq u$ also in $U_{t_1,t_2}$. 
	\end{enumerate}
\end{definition}
Note that a \lscss is defined in every point.

For the lower semicontinuity, we need to derive estimates for weak supersolutions. One of the difficulties in this is that the definition of weak supersolutions does not include a time derivative. However, we would still like to use test functions depending on the supersolution itself, and the time derivative $u_t$ inevitably appears. The forward in time mollification 
\begin{equation}
	\label{eq:naumann-conv} u^\sigma(x,t)=\frac{1}{\sigma}\int_0^t e^{(s-t)/\sigma}u(x,s)\dif s 
\end{equation}
is convenient in dealing with this defect. The aim is to establish estimates independent of the time derivative of $u^\sigma$, and then pass to the limit $\sigma\to 0$.

The basic properties of the mollification \eqref{eq:naumann-conv} are given in the following lemma, see \cite{Naumann}. 
\begin{lemma}
	\label{lem:conv-prop} 
	\begin{enumerate}
		\item If $u\in L^p(\Omega_T)$, then 
		\begin{displaymath}
			\norm{u^\sigma}_{L^p(\Omega_T)}\leq\norm{u}_{L^p(\Omega_T)}, 
		\end{displaymath}
		\begin{equation}
			\label{eq:naumann-timederiv} \frac{
			\partial u^\sigma}{
			\partial t}=\frac{u-u^\sigma}{\sigma}, 
		\end{equation}
		and $u^\sigma\to u$ in $L^p(\Omega_T)$ as $\sigma\to 0$. 
		\item If $\nabla u\in L^p(\Omega_T)$, then $\nabla(u^\sigma)=(\nabla u)^\sigma$, 
		\begin{displaymath}
			\norm{\nabla u^\sigma}_{L^p(\Omega_T)}\leq\norm{\nabla u}_{L^p(\Omega_T)}, 
		\end{displaymath}
		and $\nabla u^\sigma\to \nabla u$ in $L^p(\Omega_T)$ as $\sigma \to 0$. 
		\item If $\varphi\in C(\overline{\Omega_T})$, then 
		\begin{displaymath}
			\varphi^\sigma(x,t)+e^{-t/\sigma}\varphi(x,0)\to \varphi(x,t) 
		\end{displaymath}
		uniformly in $\Omega_T$ as $\sigma \to 0$. 
	\end{enumerate}
\end{lemma}

We need the equation satisfied by the mollification $u^\sigma$ of a weak supersolution, given by 
\begin{equation}
	\label{eq:reg-pme} \int_{\Omega_T}\varphi\frac{
	\partial u^\sigma}{
	\partial t} +\nabla (u^m)^\sigma\cdot\nabla\varphi \dif x\dif t\geq \int_{\Omega}u(x,0)\left(\frac{1}{\sigma}\int_{0}^T\varphi e^{-s/\sigma} \dif s\right)\dif x. 
\end{equation}
This equation is required to hold for all test functions $\varphi\in L^2(0,T; H^1_0(\Omega))$. This follows by straightforward manipulations involving a change of variables and Fubini's theorem.

\section{An energy estimate}

\label{sec:energy_estimate}

In this section, we derive an energy estimate for truncations of weak supersolutions. Specifically we obtain an energy estimate for level sets of $(M-u^m)_+$ of subsolution type, but since the equation does not allow addition of constants, our constants in the energy estimate depends on $M$. We use the auxiliary function in the following lemma to eliminate the time derivative when deriving the energy estimate. 
\begin{lemma}
	Let $v \geq 0$ and $m > 1$, define
	\begin{displaymath}
		B(v)=\frac{1}{m}\int_0^v(L-s)^{1/m-1}s\dif s. 
	\end{displaymath}
	Then for any nonnegative differentiable function $f(t)$ we have
	\begin{equation}
		\label{eq:B} \frac{\partial f}{\partial t}(L-f^m)_+=-\frac{\partial}{\partial t}B[(L-f^m)_+]. 
	\end{equation}
	Further, we have for any nonnegative number $v$
	\begin{equation}
		\label{eq:B2} 
		\begin{aligned}
			B[(L-v^m)_+]\leq &(L-v^m)_+(L^{1/m}-v)_+,\text{ and}\\
			B[(L-v^m)_+]\geq &L^{1/m-1}\frac{(L-v^m)_+^2}{2}. 
		\end{aligned}
	\end{equation}
\end{lemma}
\begin{proof}
  Denote $g=f^m$. We have
  \begin{align*}
    \frac{
      \partial B[(L-g)_+]}{\partial t}=& \frac{1}{m}\frac{\partial (L-g)_+ }{		\partial t} (L-(L-g)_+)^{1/m-1}(L-g)_+\\
    =& -\frac{1}{m}\frac{\partial g }{\partial t}g^{1/m-1}(L-f)_+\\
    =& -\frac{\partial g^{1/m} }{\partial t}(L-g)_+,
  \end{align*}
  which gives \eqref{eq:B}.
	
	For the first inequality in \eqref{eq:B2}, we have by an elementary estimate and computing the integral that 
	\begin{align*}
		B[(L-v^m)_+]\leq & (L-v^m)_+\frac{1}{m}\int_{0}^{(L-v^m)_+}(L-s)^{1/m-1}\dif s\\
		=&(L-v^m)_+(L^{1/m}-v)_+. 
	\end{align*}
	For the second, we use the fact that $s\mapsto (L-s)^{1/m-1}$ is increasing since $m>1$ to get 
	\begin{displaymath}
		B[(L-v^m)_+]\geq L^{1/m-1}\int_{0}^{(L-v^m)_+}s\dif s=L^{1/m-1}\frac{(L-v^m)_+^2}{2} \qedhere 
	\end{displaymath}
\end{proof}
\begin{lemma}
	\label{lem:energy} Let $u$ be a nonnegative weak supersolution for $m > 1$ in $\Omega_{t_1,t_2}$. Then the following truncated energy estimate holds for the function $v(x,t) = (M - u^m(x,t))_+$ and any $\phi \in C_0^\infty(\Omega_{t_1,t_2})$, $\phi \geq 0$, $M \geq 0$ and $k \geq 0$. 
	\begin{align}
		\label{eqenergy} \tint &|\nabla (v - k)_+|^2 \phi^2 \dxdt + \tsup ((v - k)_+)^2 \phi^2 \dx \notag \\
		&\leq C(M) \tint (v - k)_+^2 |\nabla \phi|^2 \dxdt + C \tint \phi (\phi_t)_+ ((v - k)_+) \dxdt 
	\end{align}
	where $C(M) = C (1+\max\{M^{1/m},M^{1-1/m}\})$. 
\end{lemma}
\begin{proof}
	From the nonnegativity of $u$ and \eqref{eq:reg-pme}, the mollification satisfies 
	\begin{equation} \label{eq:energy1}
		\tauint\frac{
		\partial u^\sigma}{
		\partial t}\varphi+\nabla(u^m)^\sigma\cdot\nabla \varphi\dif x\dif t\geq 0 
	\end{equation}
	for all positive test functions with compact support in space with $t_1 < \tau_1 < \tau_2 < t_2$. We take $\varphi=(L-u^m)_+\phi^2$ in this inequality, where $L\geq 0$ will be chosen later. In the time term, we write 
	\begin{align*}
		\frac{
		\partial u^\sigma}{
		\partial t}(L-u^m)_+\phi^2=& \frac{
		\partial u^\sigma}{
		\partial t}(L-(u^\sigma)^m)_+\phi^2\\
		& +\frac{u-u^\sigma}{\sigma}[(L-u^m)_+-(L-(u^\sigma)^m)_+]. 
	\end{align*}
	Since $t\mapsto (L-t^m)_+$ is decreasing, the second term is negative, and we may discard it. We continue by using \eqref{eq:B} and integration by parts to get 
	\begin{align} \label{eq:energy2}
		\tauint\frac{
		\partial u^\sigma}{
		\partial t}(L-u^m)_+\phi^2\dif x\dif t\leq& -\left.\int_{\Omega}B[(L-(u^\sigma)^m)_+]\phi^2\dif x\right|_{\tau_1}^{\tau_2}\notag\\
		&+\tauint B[(L-(u^\sigma)^m)_+]\phi (\phi_t)_+\dif x\dif t. 
	\end{align}
	
	We plug $\varphi = (L-u^m)_+ \phi^2$ into \eqref{eq:energy1}
        and 
        use the estimate \eqref{eq:energy2}, and rearrange terms. Then let
        $\sigma\to 0$, take absolute values and use Young's inequality
        to get the estimate
	\begin{align}
		\tauint &\abs{\nabla(L-u^m)_+ }^2\phi^2\dif x\dif t +\int_{\Omega}B[(L-u^m)_+]\phi^2\dif x \bigg \rvert_{\tau_1}^{\tau_2} \notag\\
		\leq & \tauint\abs{\nabla(L-u^m)_+}(L-u^m)_+\phi\abs{\nabla\phi}\dif x\dif t + \tauint B[(L-u^m)_+](\phi_t)_+\phi\dif x\dif t \notag\\
		\leq & \frac{1}{2}\tauint\abs{\nabla(L-u^m)_+}^2\phi^2\dif x\dif t+c\tauint(L-u^m)_+^2\abs{\nabla\phi}^2\dif \dif t\notag\\
		& + \tauint B[(L-u^m)_+](\phi_t)_+ \phi\dif x\dif t. \label{eq:energy3}
	\end{align}
	To continue, reabsorbing the first term on the right hand side, using \eqref{eq:B2} to estimate the third term on the right, letting $\tau_1 \to t_1$, and get since \eqref{eq:energy3} holds for all $t_1 < \tau_2 < t_2$,
	\begin{align}
		\tint\abs{\nabla(L-u^m)_+ }^2\phi^2\dif x\dif t\leq c\tint&(L-u^m)_+^2\abs{\nabla\phi}^2\dif \dif t\notag\\
		+cL^{1/m}\tint&(L-u^m)_+(\phi_t)_+ \phi \dif x\dif t.  \label{eq:energy4}
	\end{align}
	Choose $\tau_2$ such that 
	\begin{displaymath}
		\int_{\Omega}(L-u^m)_+^2\phi^2 (x,\tau_2)\dif x\geq \frac{1}{2} \tsup(L-u^m)_+^2\phi^2(x,t)\dif x. 
	\end{displaymath}
	By an application of \eqref{eq:B2}, plugging the result into \eqref{eq:energy3}, and using \eqref{eq:energy4} we get
	\begin{align*}
		\tint&\abs{\nabla (L-u^m)_+}^2\phi^2\dif x\dif t +\tsup(L-u^m)_+^2\phi^2\dif x\\
		\leq & c(1+L^{1/m}+L^{1-1/m})\tint(L-u^m)_+^2\abs{\nabla\phi}^2 +(L-u^m)_+(\phi_t)_+\phi\dif x\dif t. 
	\end{align*}
	To finish the proof, take $L=M-k$ and note that 
	\begin{displaymath}
		(M-k-u^m)_+=((M-u^m)_+-k)_+. 
	\end{displaymath}
	When $k\leq M$, we have 
	\begin{displaymath}
		(M-k)^{1/m}\leq M^{1/m} \quad\text{and}\quad (M-k)^{1-1/m}\leq M^{1-1/m}, 
	\end{displaymath}
	and the claim follows for such $k$. If $k>M$, the claim holds trivially since $u$ is nonnegative, as then $((M-u^m)_+-k)_+=0$. 
\end{proof}

\section{Local boundedness} \label{sec:boundedness} The next step is
proving an $L^\infty$ estimate by iterating the energy estimate in a
suitable way. We adapt the De Giorgi type iteration given on pp. 35-37
in \cite{DaskalopoulosKenig}, attributed in \cite{DaskalopoulosKenig}
to a personal communication of Bouillet, Caffarelli, and Fabes.

We fix a point $(x_0,t_0)\in \Omega\times (t_1,t_2)$ and use the notation 
\begin{displaymath}
	Q^-_{r,r^2}=B(0,r)\times (-r^2,0) 
\end{displaymath}
and 
\begin{displaymath}
	Q_{r,r^2}=B(0,r)\times (-r^2,r^2) 
\end{displaymath}

\begin{remark}
	The reason for the introduction of the function $G$ in Lemma \ref{lemsup} comes from the fact that we have an $L^2$ and an $L^1$ term on the right side in the energy estimate, Lemma \ref{lem:energy}, which causes problems in the De Giorgi iteration. This is evident when combining the $L^1$ estimate \eqref{vn1+} and the $L^2$ estimate \eqref{vn2} into the iteration inequality \eqref{eq:vnit}.
\end{remark}

\begin{lemma}
	\label{lemsup} Let $u,m,M$ be as in Lemma \ref{lem:energy}, and let $\sigma \in (0,1)$ be given. Denote 
	\begin{displaymath}
		G(\delta) = \max \{ \delta^{\frac{1}{N+4}}, \delta^{\frac{1}{4}} \}. 
	\end{displaymath}
	Then there exists a constant $C > 0$ such that 
	\begin{align*}
		\esssup_{(x_0,t_0)+Q^-_{\sigma \rho, \sigma \rho^2}} &(M-u^m)_+ \\
		&\leq C \bigg [\frac{C_1(M)}{(1-\sigma)^2} \bigg ]^{\frac{N+2}{4}} G \left ( \fint_{(x_0,t_0)+Q^-_{\rho,\rho^2}} (M-u^m)_+^2 + (M-u^m)_+ \dxdt \right ) 
	\end{align*}
	whenever $\rho>0$ is small enough, so that $(x_0,t_0)+Q^-_{\rho,\rho^2}\Subset \Omega_{t_1,t_2}$ and where $C_1(M) = \Big (1+\big (M^{2}+M \big )^{\frac{1}{N+2}} \Big ) \Big (1+\max\{M^{1/m},M^{1-1/m}\}\Big )$. 
\end{lemma}
\begin{remark}
	\label{remsup} Note that we can replace the cylinders $(x_0,t_0)+Q^-_{\sigma \rho, \sigma \rho^2}$ with time-symmetric cylinders, i.e. $(x_0,t_0)+Q_{\sigma \rho, \sigma \rho^2}$. 
\end{remark}
To prove Lemma \ref{lemsup} we need the following two fundamental lemmas.

\begin{lemma}
	[\cite{DiBe} p. 12] \label{lem:geomconv} Let $\{Y_n\}$, $n = 1,2,\ldots$, be a sequence of positive numbers satisfying 
	\begin{equation*}
		Y_{n+1} \leq C b^n Y_n^{1+\epsilon/2} 
	\end{equation*}
	where $C,b > 1$ and $\epsilon > 0$ are given numbers. Then if 
	\begin{equation*}
		Y_0 \leq C^{-\frac{2}{\epsilon}} b^{-\frac{4}{\epsilon^2}} 
	\end{equation*}
	we get $Y_n \to 0$ as $n \to \infty$. 
\end{lemma}
\begin{lemma}
	[\cite{DiBe} p. 9] \label{lem:paranormal-sobolev} There exists a constant $C = C(N) > 1$ such that if $u \in L^2(0,T;H^1_0(\Omega))$ then 
	\begin{equation*}
		\int_{\Omega_T} |u|^q \dxdt \leq C \left ( \int_{\Omega_T} |\grad u|^2 \dxdt \right ) \left ( \esssup_{0 < t < T} \int_{\Omega} u^2 \dx \right )^{2/N}, 
	\end{equation*}
	where $q = 2 \frac{N+2}{N}$. 
\end{lemma}

\begin{proof}[Proof of Lemma \ref{lemsup}]
	Let 
	\begin{equation*}
		\rho_n = \sigma\rho + \frac{(1-\sigma)}{2^n}\rho, \quad\text{and} \quad \theta_n = \sigma\rho^2 + \frac{(1-\sigma)}{2^n}\rho^2. 
	\end{equation*}
	We define then the corresponding cylinders 
	\begin{displaymath}
		Q_n = (x_0,t_0)+Q^-_{\rho_n,\theta_n} = B_n \times (t_0-\theta_n,t_0)\quad\text{and}\quad Q_\infty = (x_0,t_0)+Q^-_{\sigma \rho, \sigma \rho^2}, 
	\end{displaymath}
	and denote 
	\begin{equation}
		\tilde Q_n = (x_0,t_0)+Q^-_{\frac{\rho_n + \rho_{n+1}}{2},\frac{\theta_n + \theta_{n+1}}{2}}. 
	\end{equation}
	We will use the levels $k_n = k - \frac{k}{2^n}$, where the number $k$ shall be fixed later. 
	
	Take a cutoff function $\zeta$ such that 
	\begin{equation*}
		\begin{cases}
			\zeta, \text{vanishes on the parabolic boundary of $Q_n$} \\
			\zeta = 1, \text{ in $\tilde Q_n$}, \\
			|\nabla \zeta| \leq \frac{2^{n+2}}{(1-\sigma) \rho}, \quad 0 \leq \zeta_t \leq \frac{2^{n+2}}{(1-\sigma) \rho^2}, 
		\end{cases}
	\end{equation*}
	and denote $v_n = (v-k_n)_+$. We aim at deriving an estimate for the mean of $v_{n+1}^2+v_{n+1}$ in terms of the mean of $v_n^2+v_n$ so that fast geometric convergence (Lemma \ref{lem:geomconv}) can be applied. The estimate \eqref{eqenergy} gives 
	\begin{align*}
		\int_{Q_n} |\nabla v_{n+1}|^2 \zeta^2 \dxdt &+ \esssup_{t_0-\theta_n t_0< t < t_0} \int_{B_n} v_{n+1}^2 \zeta^2 \dx \\
		&\leq C(M) \left (\frac{2^{n+2}}{(1-\sigma) \rho} \right )^2 \int_{Q_n} v_{n+1}^2 \dxdt + C \frac{2^{n+2}}{(1-\sigma) \rho^2} \int_{Q_n} v_{n+1} \dxdt\\
		&\leq C \frac{2^{2n}}{(1-\sigma)^2} \frac{C(M)}{\rho^2}\int_{Q_n} v_{n+1}^2 + v_{n+1} \dxdt. 
	\end{align*}
	Take then $\tilde \zeta \in C^\infty$ such that $\tilde \zeta = 1$ in $Q_{n+1}$ and vanishes on the parabolic boundary of $\tilde Q_n$, with the same bounds for the derivatives as for $\zeta$. We use the parabolic Sobolev embedding (Lemma \ref{lem:paranormal-sobolev}) to get 
	\begin{align}
		\int_{\tilde Q_n} &v_{n+1}^q \tilde \zeta^q \dxdt \notag \\
		&\leq C \left ( \esssup_{-\tilde \theta_n < t < 0} \int_{\tilde B_n} v_{n+1}^2 \dx \right )^{2/N} \left ( \int_{\tilde Q_n} |\grad v_{n+1}|^2 \tilde \zeta^2 \dxdt + \int_{\tilde Q_{n}} v_{n+1}^2 |\grad \tilde \zeta_n|^2 \dxdt \right ) \notag\\
		&\leq \left ( \frac{2^{2n}}{(1-\sigma)^2} \frac{C(M)}{\rho^2}\int_{Q_n} v_{n+1}^2 + v_{n+1} \dxdt \right )^{q/2}, 
	\end{align}
	where $q = 2 \frac{N+2}{N}$. Denote then 
	\begin{displaymath}
		A_{n+1} = \{(x,t) \in Q_{n+1} : v_{n+1}(x,t) > 0\}. 
	\end{displaymath}
	By H\"older's inequality we get 
	\begin{align*}
		\fint_{Q_{n+1}} v_{n+1}^2 \dxdt &\leq C \left (\fint_{Q_{n+1}} v_{n+1}^q \dxdt \right )^{2/q} \left (\frac{|A_{n+1}|}{|Q_{n+1}|} \right )^{1-2/q}. 
	\end{align*}
	Further, we have 
	\begin{align}
		\frac{|A_{n+1}|}{|Q_{n+1}|} \leq C \frac{|\{[u_n>k_{n+1}-k_n = \frac{k}{2^{n+1}}] \cap Q_{n} \}|}{|Q_{n}|} \leq C \left (\frac{2^{n+1}}{k} \right )^2 \fint_{Q_{n}} v_n^2 \dxdt. 
	\end{align}
	We combine the previous estimates to get 
	\begin{align}
		\label{vn2} \fint_{Q_{n+1}} &v_{n+1}^2 \dxdt \notag\\
		&\leq C |Q_{n+1}|^{1-2/q} \left ( \frac{2^{2n}}{(1-\sigma)^2} \frac{C(M)}{\rho^2}\fint_{Q_n} v_n^2 + v_n \dxdt \right ) \left (\left (\frac{2^{n+1}}{k} \right )^2 \fint_{Q_{n}} v_n^2 +v_n \dxdt \right )^{1-2/q} \notag\\
		&\leq C \frac{C(M) 4^{n\frac{q-1}{q}}}{(1-\sigma)^2} k^{\frac{2(2-q)}{q}} \left ( \fint_{Q_{n}} v_n^2 +v_n \dxdt \right )^{1+\epsilon}. 
	\end{align}
	Note that $2-2/q = 1+2/(N+2)$.
	
	To obtain an iterative estimate we still need to bound the mean of $v_{n+1}$. To this end, we estimate 
	\begin{align}
		\label{vn1} \fint_{Q_{n+1}} v_{n+1} \dxdt &\leq \fint_{Q_{n+1}} v_{n+1} v_n \frac{2^{n+1}}{k} \dxdt \notag\\
		&\leq \frac{2^{n+1}}{k} \left ( \fint_{Q_{n+1}} v_{n+1}^2 \dxdt \right )^{1/2}\left ( \fint_{Q_{n+1}} v_n^2 \dxdt \right )^{1/2}. 
	\end{align}
	Note that in $A_{n+1}$ we know that $v_n > k_{n+1}-k_n =
        \frac{k}{2^{n+1}}$. Notice now that the first term on the
        right hand side in \eqref{vn1} is bounded by means of
        \eqref{vn2} and the second term is already essentially what we want. Thus 
	\begin{align}
		\label{vn1+} \fint_{Q_{n+1}} v_{n+1} \dxdt &\leq \frac{2^{n+1}}{k} \left ( C \frac{C(M) 4^{n\frac{q-1}{q}}}{(1-\sigma)^2} k^{\frac{2(2-q)}{q}} \right )^{1/2} \left ( \fint_{Q_{n}} v_n^2 +v_n \dxdt \right )^{1+\epsilon/2} \notag\\
		&\leq \frac{b^n C(M)}{(1-\sigma)^2} k^{2 \frac{1-q}{q}} \left ( \fint_{Q_{n}} v_n^2 +v_n \dxdt \right )^{1+\epsilon/2}. 
	\end{align}
	Let now 
	\begin{equation} \label{eq:vnit}
		Y_n = \fint_{Q_n} v_n^2+v_n \dxdt \quad\text{for}\quad n=0,1,2,\ldots. 
	\end{equation}
	To counter the discrepancy between the power $1+\epsilon/2$ in
        \eqref{vn1+} and $1+\epsilon$ in \eqref{vn2} we note that $v_n
        \leq v \leq M$ and we get 
	\begin{equation*}
		Y_{n+1} \leq C \left (1+\frac{1}{k^{1/q}} \right ) \frac{C_1(M) b^{n}}{(1-\sigma)^2} k^{\frac{2(2-q)}{q}} Y_n^{1+\epsilon/2}.
	\end{equation*}
	Here $\epsilon = 2/(N+2)$,  $b = b(N) > 1$, and $C_1(M) = \Big [1+\big [M^{2}+M \big ]^{\epsilon / 2} \Big ] C(M)$. By fast geometric convergence (Lemma \ref{lem:geomconv}) we see that if 
	\begin{equation}
		\label{eq:iteriter} Y_0 \leq \frac{1}{C} \left ( \frac{C_1(M)}{(1-\sigma)^2} \left (1+\frac{1}{k^{1/q}} \right ) k^{\frac{2(2-q)}{q}} \right)^{-2/\epsilon}, 
	\end{equation}
	we have 
	\begin{displaymath}
		\fint_{Q_\infty} (v-k)_+^2+(v-k)_+\dif x\dif t=\lim_{n\to \infty} Y_n=0. 
	\end{displaymath}
	Thus we obtain the estimate 
	\begin{equation*}
		v \leq (v-k)_++k\leq k, 
	\end{equation*}
	almost everywhere in $Q_{\infty}$. The right hand side in \eqref{eq:iteriter} is increasing in $k$, so we see that there exists a unique $k$ such that 
	\begin{equation}
		\label{eqY0k} Y_0 = \frac{1}{C} \left ( \frac{C_1(M)}{(1-\sigma)^2} \left (\frac{k^{1/q}+1}{k^{1/q}} \right ) k^{\frac{2(2-q)}{q}} \right)^{-2/\epsilon}. 
	\end{equation}
	To relate this value of $k$ to the size of $Y_0$ we do as
        follows. First rewrite \eqref{eqY0k} as
	\begin{equation*}
		\left (\frac{k^{1/q}}{k^{1/q}+1} \right ) k^{2\epsilon} = C \frac{C_1(M)}{(1-\sigma)^2} Y_0^{\epsilon/2}. 
	\end{equation*}
	Then $k$ is bounded from above by $Y_0$ as follows. If $0 \leq k < 1$ then 
	\begin{equation*}
		k^{2 \epsilon + \frac{1}{q}} \leq C \frac{C_1(M)}{(1-\sigma)^2} Y_0^{\epsilon/2}, 
	\end{equation*}
	giving 
	\begin{equation}
		\label{eqksmall} k \leq C \left (\frac{C_1(M)}{(1-\sigma)^2} Y_0^{\epsilon/2}\right )^{\frac{1}{1+\epsilon}}. 
	\end{equation}
	On the other hand if $k \geq 1$ then we get 
	\begin{equation}
		\label{eqklarge} k \leq C \left (\frac{C_1(M)}{(1-\sigma)^2} Y_0^{\epsilon/2} \right )^{\frac{1}{2\epsilon}}. 
	\end{equation}
	The two cases \eqref{eqksmall} and \eqref{eqklarge} give
	\begin{equation*}
		k \leq C \left ( \frac{C_1(M)}{(1-\sigma)^2} \right )^{\frac{1}{2\epsilon}} \max \left \{ Y_0^{\frac{1}{N+4}},\, Y_0^{\frac{1}{4}} \right \}.
	\end{equation*}
	Recalling that $G(\delta) =
        \max\{\delta^{\frac{1}{N+4}},\delta^{\frac{1}{4}}\} $, the
        lemma follows.
\end{proof}

\section{Proof of Theorem \ref{thm:lsc:intro}} 

\label{sec:lower_semicontinuity}

Lower semicontinuity of weak supersolutions is now a fairly straightforward consequence of Lemma \ref{lemsup}. We define the \emph{lower semicontinuous regularization} $v^\ast$ of a function $v$ by 
\begin{equation*}
	v^\ast(x,t):=\essliminf_{(y,s) \to (x,t)} v(y,s) = \lim_{r \to 0} \essinf_{(x,t)+Q_{r,r^2}} v. 
\end{equation*}
An elementary argument shows that $v^\ast$ is lower semicontinuous. 

It is enough to prove that the function $u^m$, has a lower semicontinuous representative. To see this note that since $f(x) = x^m$ is strictly increasing and continuous we get that 
\begin{equation*}
	\essinf_{(x,t)+Q_{r,r^2}} u^m = (\essinf_{(x,t)+Q_{r,r^2}} u)^m 
\end{equation*}
and 
\begin{equation*}
	\essliminf_{(y,s) \to (x,t)} u^m = (\essliminf_{(y,s) \to (x,t)} u)^m. 
\end{equation*}
Therefore we prove the following version of Theorem \ref{thm:lsc:intro}.
\begin{theorem}
	\label{thm:lsc} Let $u$ be a nonnegative weak supersolution for $m > 1$ in $\Omega_{t_1,t_2}$, and let 
	\begin{displaymath}
		v=u^m. 
	\end{displaymath}
	Then 
	\begin{displaymath}
		v(x,t)=v^\ast(x,t) 
	\end{displaymath}
	at all Lebesgue points of $v$ such that $v(x,t)<\infty$. In particular, $v$ has a lower semicontinuous representative. 
\end{theorem}
\begin{proof}
	Let $E$ be the set of Lebesgue points of $v=u^m$, i.e. 
	\begin{equation*}
		E = \left \{(x,t) \in \Omega_{t_1,t_2}: \lim_{r \to 0} \fint_{(x,t)+Q_{r,r^2}} |u^m(x,t)-u^m(y,s)| dy ds = 0 \right \}. 
	\end{equation*}
	Further, let
        \begin{displaymath}
          O = \{(x,t) \in \Omega_{t_1,t_2}: v(x,t) < \infty \}.
        \end{displaymath}
  We wish to show that if $(x_0,t_0) \in E \cap O$ then $v^\ast(x_0,t_0) = v(x_0,t_0)$. Note that by the summability of $v$ we get that $|E \cap O| = |\Omega_{t_1,t_2}|$.
	
	First of all if $(x_0,t_0) \in E \cap O$ then 
	\begin{equation*}
		v^\ast (x_0,t_0) \leq \lim_{r \to 0} \fint_{(x_0,t_0)+Q_{r,r^2}} v \dxdt=v(x_0,t_0). 
	\end{equation*}
	
	For the other inequality, observe first that 
	\begin{displaymath}
		v(x_0,t_0)-\essinf_{(x_0,t_0)+Q_{\sigma r,\sigma r^2}}v=\esssup_{(x_0,t_0)+Q_{\sigma r,\sigma r^2}}(v(x_0,t_0)-v)\leq \esssup_{(x_0,t_0)+Q_{\sigma r,\sigma r^2}}(v(x_0,t_0)-v)_+. 
	\end{displaymath}
	Then take $M = v(x_0,t_0)$ in Lemma \ref{lemsup} to get 
	\begin{align}
		\label{eq1} \esssup_{(x_0,t_0)+Q_{\sigma r,\sigma r^2}} &(v(x_0,t_0)-v)_+ \notag \\
		&\leq C \bigg [\frac{C_1(M)}{(1-\sigma)^2} \bigg ]^{\frac{N+2}{4}} G \left ( \fint_{(x_0,t_0)+Q(r,r^2)} (v(x_0,t_0)-v)_+^2 + (v(x_0,t_0)-v)_+ \dxdt \right ) 
	\end{align}
	for $r > 0$ small enough, such that $(x_0,t_0)+Q(r,r^2) \Subset \Omega_{t_1,t_2}$. We have 
	\begin{displaymath}
		\fint_{(x_0,t_0)+Q_{r,r^2}} (v(x_0,t_0)-v)_+^2\dif x\dif t\leq v(x_0,t_0) \fint_{(x_0,t_0)+Q_{r,r^2}}(v(x_0,t_0)-v)_+\dif x\dif t. 
	\end{displaymath}
	Since $(x_0,t_0)$ is a Lebesgue point of $v$, we also get 
	\begin{displaymath}
		\fint_{(x_0,t_0)+Q_{r,r^2}}(v(x_0,t_0)-v)_+\dif x\dif t\leq \fint_{(x_0,t_0)+Q_{r,r^2}}\abs{v(x_0,t_0)-v}\dif x\dif t \to 0 \quad \text{as}\quad r\to 0. 
	\end{displaymath}
	Recalling the fact that $v(x_0,t_0)<\infty$, the previous two inequalities imply that the right hand side in \eqref{eq1} tends to zero as $r \to 0$. It follows that 
	\begin{displaymath}
		v(x_0,t_0)-v^\ast (x,t)\leq 0, 
	\end{displaymath}
	as desired. 
\end{proof}


\end{document}